\documentclass{amsart}
\usepackage{amsthm}

\newtheorem{thm}{Theorem}

\newtheorem{remark}{Remark}

\newtheorem{theorem}{Theorem}
\newtheorem{lemma}{Lemma}
\newtheorem{definition}[theorem]{Definition}
\newtheorem{fact}[thm]{Fact}

\author{TONG CHENG, ZHIHAN GAO, YUXIN MA, YUHAN NING,\\AND JIANGHAO XU }
\title{A  STUDY ON NICE OPEN COVERS IN CONSTRUCTIVE ANALYSIS}
\date{}
\begin{document}
\maketitle

\begin{abstract}
Mathematicians like Bishop~\cite{Bishop} and Markov~\cite{Markov} made an effort to develop constructive mathematics and extended many theorems in classical mathematical analysis. Heine Borel theorem tells us that a closed bounded subset of Euclidean space R is compact, but in constructive mathematics, Tseitin and Zaslavskii~\cite{Zaslavskii1, Zaslavskii2, TseitinZaslavskii} showed that the set of all constructive real numbers between $0$ and $1$ is not compact. We are going to show that when give certain restriction to the open cover on $[0,1]$, we can however always choose a finite sub-cover.
   
\end{abstract}

\leftline{\em \Small Keyword: constructive topology}

\section{Introduction}

In order to resolve these questions which are mentioned in the abstract, we will need some definitions and facts from the topology subject. The definitions and the facts are described below.

\begin{definition}
A subset A is {\it dense\/} in a topological space X if $\forall$  open set U$\subseteq$X $\exists$ a$\in$A s.t. a$\in$U 	
\end{definition}

\begin{definition}
A {\it open cover\/} of a set X is a collection of open sets whose union includes set X. An {\it open set\/} generalize the idea of an open interval.\end{definition}

\begin{definition}
 A constructive sequence of rational numbers CSRN is \emph{fundamental} if $\forall n \in \mathbb{N} \exists m\in \mathbb{N}$ s.t. $\forall i , j \geq m$ we have $| \alpha(i) - \alpha(j) | < 2^{-n}$. See~\cite{Kushner}
\end{definition}

\begin{definition}
 A {\it constructive real number\/} is formed by a computer algorithm $\alpha$ which is fundamental and a regulator which is a program $\beta$ given by computer algorithms that given $n$ produces $m$ as in the formula. See~\cite{Kushner}
\end{definition}

\begin{definition}
 A metric space is \emph{complete} if every Cauchy convergent sequence of points in it actually does converge.
\end{definition}
   
\begin{definition}
 A topological space $(X,\Omega)$ is called \emph{separable} if there exist a countable every where dense set $\{x_i\}_{i=1}^\infty$ of point of $X$.
\end{definition}
       
\begin{definition}
 A list $\{M,\rho\}$, where $M$ is some set of constructive objects and $\rho$ is an algorithm converting any pair of elements of $M$ into a constructive non-negative real number, is called a \emph{constructive metric space} if for any $X,Y,Z\subset M$ the following properties hold: 1) $\rho(X,X)=0$; 2) $\rho(X,Y)\le\rho(X,Z)+\rho(Y,Z)$. (The distance function $\rho$ should be symmetric i.e. $\rho(X,Y)=\rho(Y,X)$ and non-degenerate, i.e. of $\rho(X,Y)=0$, then $X=Y$. And the space is constructed from the initial countable dense set.) See~\cite{Kushner}
\end{definition}
       
\begin{definition}
 A subset $A$ of natural numbers is \emph{decidable} if there is an algorithm that given every natural number $n$ determines if $n$ is in $A$ or not. See~\cite{Kushner}
\end{definition}
       
\begin{definition}
 We say an open cover $\omega$ on interval $[0,1]$ is \emph{nice}, if $\forall CRN x \in [0,1]$, we can algorithmically choose a rational number $r \in \mathbb{Q}$, s.t. $\exists W \in \omega$, with $(x-r,x+r) \subset W$.
\end{definition}

\begin{definition}
 Suppose $X$ is a metric space with a metric $d$, and suppose $S$ is a subset of $X$. Let $\varepsilon$ be a positive real number. A subset $N\subset S$ is an $\varepsilon-net$ for $S$ if, for all $x\in S$, there is an $y\in N$, such that $d(x,y)<\varepsilon$.
\end{definition}

\begin{fact}There does exist an enumerable undecidable set.
\end{fact}

\begin{proof}
Take a computable function $f(x)$ that does not admit a computable everywhere defined extension. The domain $F$ of $f$ is enumerable, which has been proved.(See~\cite{Rogers}) If $F$ was decidable, then the function $g(x)$ defined by: $f(x)$ if $x\in F$; or 0 if $x\notin F$ will be a computable everywhere defined extension of the function f. (When computing $g(x)$, we should check if $x$ belongs to $F$ firstly, and if it does, then we compute $f(x)$.) See~\cite{Mendelson}.
\end{proof}
       
\begin{fact}
 Algorithms can be performed in steps.
\end{fact}

\section{Lemmas Used in the Main Proof of Theorem}

\begin{lemma}\label{Lemma1}

Let X be a complete separable metric space and Y be a set where detecting the equality of two elements is a decidable problem. Let $f$: X $\rightarrow$ Y be a function that is constant on a countable dense P $\supset$ X.\\Then f is a constant function.

\end{lemma}

\begin{proof} 
Let A be an algorithm with an undecidable domain set that runs in steps. We choose any x $\in$ X which is separable s.t. we construct a sequence x$_n$ $\rightarrow$ x\ where x$_n$ $\in$ P. Now we construct the new sequence using the algorithm A as follow: \\Run A on n for 1 second. If it does not terminate we write x$_1$. Then run another second (2 seconds for total). If A does not terminate we write x$_2$. If A terminates on the input n say in k seconds, we continue to write x$_k$ forever from that point.\\ Now we can deduce that if A terminates, the new sequence will converge to x$_k$ $($k is the time when A terminates on n$)$. That is, the limit of the new sequence will be x$_k$; if A never terminate, then the limit of the new sequence will be x.\\ Now we construct a new program g(x) which is exactly the constructive number that is the limit of the sequence program introduced before. Then we look at f(g(n)) where n is the input we fit into A. If f is not a constant function, then we can choose x $\in$ X s.t. f(x) $\neq$ f(p) where p $\in$ P.\\ So we can algorithmically decide if A terminates on n, which contradicts to our assumption that A has an undecidable domain set. So f is a constant function.
\end{proof}

\begin{lemma}\label{Lemma2}

Let $f: X \rightarrow Y,$ where $X$ is a constructive interval of real numbers and $Y$ as in Lemma~\ref{Lemma1}. Then $f$: X $\rightarrow$ Y is a constant function.
\end{lemma}

\begin{proof}
Since the equality of two elements is a decidable question in Y, we can put a discrete metric space on the interval with the distance function $\rho$(x,y):

$$ \rho(x,y)=\left\{
\begin{array}{rcl}
0 & & {if\ x=y}\\
1 & & {if\ x\neq y}\\
\end{array} \right. $$

Notice that the set of all the rational number in the interval, denoted by Z, is dense. If $f$ is constant on $Z,$ using Lemma~\ref{Lemma1} we know that it is also constant on the whole interval. Now assume that the statement of Lemma~\ref{Lemma2} is wrong, we can learn that $\exists p,q \in Z,$ s.t.$ f(p) \neq f(q).$
\\Consider an algorithm $A(n)$ (n$\in$$\mathbb{N}$) for generating pairs of points: Firstly, A generates the pair $(p_{0},q_{0}).$ In the first second, A generates $(p_{1},q_{1})$ while:

$$ 
(p_{1},q_{1})
\begin{cases}=
(\frac{p_0+q_0}{2}, q_0 ) &  \text{ if }\ f(\frac{p_0+q_0}{2})\neq f(q_0)\\
(p_{0},\frac{p_0+q_0}{2}) & \text{ if } f(\frac{p_0+q_0}{2})\neq f(p_0)
\end{cases}  $$

Likely, for the case in the n seconds, $A$ generates $(p_{n},q_{n}$) while\\

$$ (p_n,q_n)=\begin{cases}
(\frac{p_{n-1}+q_{n-1}}{2},q_{n-1} ) & \text{ if } f(\frac{p_{n-1})+q_{n-1})}{2})\neq f(q_{n-1})
\\
(p_{n-1},\frac{p_{n-1}+q_{n-1}}{2}) &  \text{ if }f(\frac{p_{n-1}+q_{n-1}}{2})\neq f(p_{n-1})
\end{cases} $$

In this process we make sure that in each pair $f(p_{n})\neq f(q_{n}).$ We have actually also constructed two sequences $\{p_i\}$ and $\{q_i\}.$ \\Since $${ \lim_{n \to +\infty} q_{i}-p_{i}=0},$$ using calculus we can learn that these two sequences are both fundamental, and will converge to the same CRN denoted by x.
\\Take a sequence y$_n$ in Y. Since the equality relation of elements in Y is decidable, we have an algorithm B that maps Y to the discrete metric space Z of the same cardinality with $B(y) \neq B(y')$ for $y \neq y'.$ Now we look at the composition 
$B \circ f:Y \to Z.$ According to Tseitin's theorem~\cite{Tseitin1, Tseitin2} we learn that this composition is continuous, thus the sequences $f(p_n)$ and $f(q_n)$ will converge to the same element $f(x).$ Since we have a discrete metric on 
$Y, \exists N_p$ s.t. $\forall m,n > N_p, f(p_m)= f(q_n)= f(x),$ and such $N_q$ also exists for $\{q_i\}$.
\\Denote $N=\max (N_p, N_q).$ we can see that $\forall n>N, f(p_n)= f(x)= f(q_{n}),$ which contradicts to the process of the algorithm A. Thus we show that f is constant on Z. Using Lemma~\ref{Lemma1} we can learn that f is constant on the whole interval.
\end{proof}

\section{Main Theorem}

Tseitin and Zaslavskii~\cite{Zaslavskii1, Zaslavskii2, TseitinZaslavskii} showed that the set of all CRN between 0 and 1 is not compact. That is to say, we cannot always find a finite sub-cover from its open cover. (Note that this is very different from the situation in the inuitionistic logic where such an interval is compact, see Cederquist and Negri~\cite{CN}.)

However, when we give the open cover some restriction, we will see that such finite open cover can always be found.\\

\begin{thm}
if an open cover on the set of all the CRN between 0 and 1 is nice, we can always algorithmically find a finite sub-cover.
\end{thm}

\begin{proof} Denote $[0,1]_R= \{ X-CRN | 0\leq x < 1\}$ where $X$ is the limit of the Cauchy sequence generated by programs and the regulator of the Cauchy sequence. Assume we call a program $E: X \rightarrow r$ (from the definition of a nice cover). From Lemma~\ref{Lemma2} we know that $E(x_1)= E(x_2), \forall x_1, x_2 \in [0,1]_R$.\\Since the open cover of $[0,1]_R$ is assumed to be nice, we can algorithmically choose $r$ s.t. the ball $B_r(x)$ with radius $r$ is contained in some elements of the open cover. Since $r$ is same for all $x$ and the number of $\epsilon$-net where $\epsilon=r$ is finite by definition, we can let the points of $\epsilon$-net be the centre of the balls $B_r(x).$ So we can get a finite number of 
$B_r(x).$\\ For each ball $B_r(x),$ there is possibly a infinite number of open set from the open cover $[0,1]_R$. So we can algorithimically choose one of these elements of open covers while a ball $B_r(x)$ is inside it. So for a nice open cover, We can always find a finite open sub-cover. 
\end{proof}
\begin{remark}
Of course the same fact remains true when instead of the closed constructive interval [0,1]$_R$ one considers any path connected, complete and separable constructive metric space.
\end{remark}

{\bf Acknowledgement:}
This paper was written as a result of the research project that the authors undertook at the CIS Online Education Company organized by the Torhea Education Group Inc. We are thankful to Torhea for organizing this program and to Vladimir Chernov for supervising the research project, who in turn is very grateful to Viktor Chernov for the multiple discussions about the problem and the research field in general.

\end{document}